\documentclass[a4paper,11pt]{article}
\usepackage{amsmath}
\usepackage{amsthm}
\usepackage{color}
\usepackage{amssymb}
\usepackage[numbers]{natbib}
\usepackage{graphicx}
\usepackage{multirow}

\usepackage{fancyhdr} 
\usepackage[top=1in, bottom=1in, left=1in, right=1in]{geometry} 
\usepackage{scrextend} 
\usepackage[notablist]{endfloat} 

\newfont{\popis}{cmcsc10}

 \newtheorem{defi}{Definition} 
 \newtheorem{thm}{Theorem} 
 \newtheorem{cor}{Corollary}  
\newtheorem{lem}{Lemma} 
\newtheorem{rmk}{Remark} 
\theoremstyle{definition}

 \newcommand{\carka}{\raise0.2em\hbox{,}}

\newcommand{\begeqO}{\begin{eqnarray*}}
\newcommand{\eneqO}{\end{eqnarray*}}
\newcommand{\begeq}{\begin{eqnarray}}
\newcommand{\eneq}{\end{eqnarray}}

\newcommand{\nin}{\noindent}

\newcommand{\cL}{\mathcal{L}}
\newcommand{\cS}{\mathcal{S}}

\newcommand{\Ima}{\mathrm{Im}}

\newcommand{\spa}{\overline{\mathrm{sp}}}

\newcommand{\convP}{\stackrel{P}{\rightarrow}}
\newcommand{\isd}{\stackrel{d}{=}}

\newcommand{\E}{E}

\newcommand{\rank}{\mathrm{rank}}
\newcommand{\Prob}{P}
\newcommand{\ip}[2]{\langle #1,#2 \rangle}

\begin{document}
\renewcommand{\thefootnote}{\fnsymbol{footnote}}
\nin {\LARGE\textbf{A note on estimation in Hilbertian linear models}
\renewcommand{\thefootnote}{\arabic{footnote}}}

\makeatletter
\renewcommand{\@fnsymbol}[1]{\ensuremath{%
   \ifcase#1\or  \star \or \ddagger \or *\or
   \mathsection\or \mathparagraph\or \|\or \star\or
   \star\star\or {\star\star}\star \else\@ctrerr\fi}}
\makeatother

\bigskip\nin
{
{\textbf{Siegfried H\"ormann, {\L}ukasz Kidzi\'nski}}

\bigskip\nin
{
D\'epartment de Math\'ematique,
Universit\'e libre de Bruxelles (ULB),
Belgium}

\bigskip\nin

\begin{abstract}
\nin 
We study estimation and prediction in linear models where the response and the regressor variable both take values in some Hilbert space. Our main objective is to obtain consistency of a principal components based estimator for the regression operator under minimal assumptions. In particular, we avoid some inconvenient technical restrictions that have been used throughout the literature.
We develop our theory in a time dependent setup which comprises as important special case the autoregressive Hilbertian model.
\end{abstract}



\section{Introduction}\label{sec 1}

In this paper we are concerned with a regression problem of the form
\begin{equation}\label{eq:model}
Y_k=\Psi(X_k)+\varepsilon_k,\quad k\geq 1,
\end{equation}
where $\Psi$ is a bounded linear operator mapping from space $H_1$ to $H_2$. This model is fairly general and many special cases have been intensively studied in the literature. Our main objective is the study of this model when the regressor space $H_1$ is infinite dimensional. Then model \eqref{eq:model} can be seen as a general formulation of a {\em functional linear model}, which is an integral part of functional data literature.  Its  various
forms are introduced in  Chapters 12--17
of Ramsay and Silverman~\cite{rs05}. A few recent references 
are
Cuevas et al.~\cite{cuevas:febrero:fraiman:2002}, 
Malfait and Ramsay~\cite{malfait:ramsay:2003}, 
Cardot et al.~\cite{cardot:ferraty:sarda:2003},
Chiou et al.~\cite{chiou:muller:wang:2004},
M\"uller and Stadtm\"uller~\cite{muller:stadtmuller:2005},
Yao et al.~\cite{yao:muller:wang:2005},
Cai and Hall~\cite{cai:hall:2006},
Li and Hsing~\cite{li:hsing:2007},
Hall and Horowotiz~\cite{hall:horowitz:2007},
Reiss and Ogden~\cite{reiss:ogden:2007},
Febrero-Bande et al.~\cite{febrero:galeano:gonzalez:2010},
Crambes et al.~\cite{crambes:kneip:sarda:2010},
Yuan and Cai~\cite{yuan:cai:2011},
Ferraty et al.~\cite{ferraty:laksaci:tadj:vieu:2011},
Crambes and Mas~\cite{crambes:mas:2012}.

From an inferential point of view, a natural problem is the estimation of the `regression operator' $\Psi$. Once an estimator $\hat\Psi$ is obtained, we can use it in an obvious way for prediction of the responses $Y$. Both, the estimation and the prediction problem are addressed in this paper. In existing literature, these problems have been discussed from several angles. For example, there is the distinction between the `functional regressors and responses' model (e.g., Cuevas et al.~\cite{cuevas:febrero:fraiman:2002}) or the perhaps more widely studied `functional regressor and scalar response model' (e.g., Cardot et al.~\cite{cardot:ferraty:sarda:1999}). Other papers deal with the effect when random functions are not fully observed but are obtained from sparse, irregular data measured with error (e.g., Yao et al.~\cite{yao:muller:wang:2005}). More recently, the focus was on establishing rates of consistency (e.g., Cai and Hall~\cite{cai:hall:2006}, Cardot and Johannes~\cite{cardot:johannes:2010}). The two most popular methods of estimation are based on principal component analysis (e.g., Bosq~\cite{bosq:1991}, Cardot et al.~\cite{cardot:ferraty:sarda:1999}, Hall and Horowitz~\cite{hall:horowitz:2007}) or spline smoothing estimators (e.g., Hastie and Mallows~\cite{hm93}, Marx and Eiler~\cite{me96}, Crambes et al.~\cite{crambes:kneip:sarda:2010}).

In this paper we address the estimation and prediction problem for this model when the data are fully observed, using the principal component (PC) approach.
Let us explain what is the new contribution and what distinguishes our paper from previous work.\medskip

\nin
(i) The crucial difficulty for this type of problems is that the infinite dimensional operator $\Psi$ needs to be approximated by a sample version $\hat\Psi_K$ of finite dimension $K$, say. Clearly, $K=K_n$ needs to depend on the sample size and tend to $\infty$ in order to obtain an asymptotically unbiased estimator. In existing papers determination of $K$ and proof of consistency require, among others, unnecessary moment assumptions and artificial restrictions concerning the spectrum of the covariance operator of the regressor variables $X_k$. As our main result, we will complement the current literature by showing that the PC estimator remains consistent without such technical constraints. We provide a {\em data-driven procedure} for the choice of $K$, which may even be used as a practical alternative to cross-validation.\medskip

\nin
(ii) We allow the regressors $X_k$ to be dependent. This is important for two reasons. First, many examples in FDA literature exhibit dependencies as the data stem from a continuous time process, which is then segmented into a sequence of curves, e.g., by considering daily data. Examples of this kind include intra-day patterns of pollution records, meteorological data, financial transaction data or sequential fMRI recordings. See, e.g., Horv\'ath and Kokoszka~\cite{horvath:kokoszka:2012}.

 Second, our framework detailed below will include the important special case of a {\em functional autoregressive model} which has been intensively investigated in the functional literature and is often used to model autoregressive dynamics of a functional time series. This model is analyzed in detail in Bosq~\cite{b00}.  We can not only greatly simplify the assumptions needed for consistent estimation, but also allow for a more general setup. E.g., in our Theorem~\ref{thm:convX} we show that it is not necessary to assume that $\Psi$ is a Hilbert-Schmidt operator if our intention is prediction. This  quite restrictive assumption is standard in existing literature, though it even excludes the identity operator.\medskip

\nin
(iii) As we already mentioned before, the literature considers different forms of functional linear models. Arguably the most common are the {\em scalar response and functional regressor} and the {\em functional response and functional regressor case}. We will not distinguish between these cases, but work with a linear model between two general Hilbert spaces.\medskip

In the next section we will introduce notation, assumptions, the estimator and our main results. In Section~\ref{s:simulation} we provide a small simulation study which compares our data driven choice of $K$ with cross-validation (CV). As we will see, this procedure is quite competitive with CV in terms of mean squared prediction error, while it is clearly favorable to the latter in terms of computational costs. Finally, in Section~\ref{s:proofs}, we give the proofs.

\section{Estimation of $\Psi$}\label{sec 2}

\subsection{Notation}
Let $H_1,H_2$ be two (not necessarily distinct) separable Hilbert spaces. We denote by $\cL(H_i,H_j)$, $(i,j\in\{1,2\}$), the space of bounded linear operators from $H_i$ to $H_j$. Further we write $\ip{\cdot}{\cdot}_{H}$ for the inner product on Hilbert space $H$ and $\|x\|_{H}^2 = \ip{x}{x}_{H}$ for the corresponding norm. For $\Phi \in \cL (H_i,H_j)$ we denote by 
$\|\Phi\|_{\cL(H_i,H_j)} = \sup_{\|x\|_{H_i} \leq 1} \| \Phi(x) \|_{H_j}$ the operator norm and by $\|\Phi\|^2_{\mathcal{S}(H_i,H_j)} = \sum_{k=1}^\infty \|\Phi(e_k)\|_{H_j}^{2}$, where $e_1,e_2,... \in H_i$ is any orthonormal basis (ONB) of $H_i$,  the Hilbert-Schmidt norm of $\Phi$. It is well known that this norm is independent of the choice of the basis. Furthermore, with the inner product $\langle \Phi,\Theta\rangle_{\cS(H_1,H_2)}=\sum_{k\geq 1}\langle \Phi(e_k), \Theta(e_k)\rangle_{H_2}$ the space $\cS(H_1,H_2)$ is again a separable Hilbert space. For simplifying the notation we use $\cL_{ij}$ instead of $\cL(H_i,H_j)$ and in the same spirit $\cS_{ij}$, $\|\cdot\|_{\cL_{ij}}$, $\|\cdot\|_{\cS_{ij}}$ and $\langle\cdot,\cdot\rangle_{\cS_{ij}}$. 

All random variables appearing in this paper will be assumed to be defined on some common probability space $(\Omega,\mathcal{A},P)$. A random element $X$ with values in $H$ is said to be in $L_{H}^p$ if $\nu_{p,H}(X) := (E\|X\|_{H}^p )^{1/p} < \infty$. More conveniently we shall say that $X$ has $p$ moments. If $X$ possesses a first moment, then $X$ possesses a mean $\mu$, determined as the unique element for which $E\langle X,x\rangle_H=\langle \mu,x\rangle_H$, $\forall x\in H$. For $x \in H_i$ and $y \in H_j$ let $x \otimes y : H_i \rightarrow H_j$ be an operator defined as $ x \otimes y (v) = \ip{x}{v} y$.  If $X\in L_H^2$, then it possesses a covariance operator $C$, given by $C=E[(X-\mu)\otimes (X-\mu)]$. It can be easily seen that $C$ is a Hilbert-Schmidt operator. Assume $X,Y\in L_H^2$. Following Bosq~\cite{b00}, we say that $X$ and $Y$ are orthogonal ($X\perp Y$) if $EX\otimes Y=0$. A sequence of orthogonal elements in $H$ with a constant mean and constant covariance operator is called {\em $H$--white noise.}

\subsection{Setup}\label{s:setup}

We consider the general regression problem \eqref{eq:model} for {\em fully observed data}. Let us collect our main assumptions.\medskip

\noindent
{\bf (A):} {\em We have $\Psi\in \cL_{12}$. Further $\{\varepsilon_k\}$ and $\{X_k\}$ are zero mean variables which are assumed to be {\em $L^4$--$m$--approximable} in the sense of H\"ormann and Kokoszka~\cite{hk10} (see below). In addition $\{\varepsilon_k\}$ is $H_2$--white noise. For any $k\geq 1$ we have $X_k\perp \varepsilon_k$.}\medskip

Here is the weak dependence concept that we impose.
%
\begin{defi}[H\"ormann and Kokoszka~\cite{hk10}]\label{def:mapr}
A random sequence $\{X_n\}_{n\geq1}$ with values in $H$ is called $L^p$--$m$--approximable, if it can be represented as 
\begin{align*}
X_n = f(\delta_n,\delta_{n-1},\delta_{n-2},...),
\end{align*}
where the $\delta_i$ are iid elements taking values in a measurable space $S$ and $f$ is a measurable function $f : S^\infty \rightarrow H$. Moreover, if $\delta'_i$ are independent copies of $\delta_i$ defined on the same probability space, then for
\begin{align*}
X_n^{(m)} = f(\delta_n,\delta_{n-1},\delta_{n-2},...,\delta_{n-m+1},\delta'_{n-m},\delta'_{n-m-1},...)
\end{align*}
we have
\begin{align*}
\sum_{m=1}^\infty \nu_{p,H}(X_m - X_m^{(m)}) < \infty.
\end{align*}
\end{defi}

Evidently, i.i.d.\ sequences with finite $p$-th moments are $L^p$--$m$--approximable. This leads to the classical functional linear model.  But it is also easily checked that functional linear processes fit in this framework. More precisely, if $X_n$ is of the form
$$
X_n=\sum_{k\geq 0}b_k(\delta_{n-k}),
$$ 
where $b_k:H_0\to H_1$ are bounded linear operators such that $\sum_{m\geq 1}\sum_{k\geq m}\|b_k\|_{\mathcal{L}_{01}}<\infty$, and $(\delta_n)$ is i.i.d.\ noise with $\nu_{p,H_0}(\delta_0)<\infty$, then $\{X_n\}$ is $L^p$--$m$--approximable. Other (also non-linear) examples of functional time series covered by $L^p$--$m$--approximability can be found in \cite{hk10}. 

A very important example included in our framework is the {\em autoregressive Hilbertian model of order 1} (ARH(1)) given by the recursion $X_{k+1}=\Psi(X_k)+\varepsilon_{k+1}$. It will be treated in more detail in Section~\ref{ss:results}. 

The notion of $L^4$--$m$--approximability implies that the process is stationary and ergodic and that it has finite forth moments. The latter is in line with existing literature. We are not aware of any article that works with less than 4 moments. In contrast, for several consistency results finite moments of all orders (or even bounded random variables) are assumed. Since our estimator below is a moment estimator, based on second order moments, one could be tempted to believe that some of our results may be deduced directly from the ergodic theorem under finite second moment assumptions. We will explain in the next section, after introducing the estimator, why this line of argumentation is not working.

Our weak dependence assumption implies that a possible non-zero mean of $X_k$ can be estimated consistently by the sample mean. Moreover we have (see \cite{hk12})
$$
\sqrt{n}\|\bar{X}-\mu\|_{H_1}=O_P(1).
$$
We conclude that the mean can be accurately removed in a preprocessing step and that $EX_k=0$ is not a stringent assumption. Since by Lemma~2.1 in \cite{hk10} $\{Y_k\}$ will also be $L^4$--$m$--approximable, the same argument justifies that we study a linear model without intercept.

\subsection{The estimator}
\label{ss:estimator}
The PC based estimator for $\Psi$ described below was first studied by Bosq~\cite{bosq:1991} and is based on a finite basis approximation. To achieve optimal approximation in finite dimension, one chooses eigenfunctions of the covariance operator $C= \E [X_1 \otimes X_1]$ as a basis. 
Let $\Delta = \E[ X_1 \otimes Y_1]$. By Assumption {\bf (A)} both, $\Delta$ and $C$, are Hilbert-Schmidt operators.
Let ($\lambda_i, v_i)_{i \geq 1}$ be the eigenvalues and corresponding eigenfunctions of the operator $C$, such that $\lambda_1 \geq \lambda_2 \geq ... $. The eigenfunctions are orthonormal and those belonging to a non-zero eigenvalue form an orthonormal basis of $\overline{\Ima(C)}$, the closure of the image of $C$. Note that, with probability one, we have $X\in\overline{\Ima(C)}$. Since $\overline{\Ima(C)}$ is again a Hilbert-space, we can assume that $H_1=\overline{\Ima(C)}$, i.e.\ that the operator is of full rank. In this case all eigenvalues are strictly positive. Using linearity of $\Psi$ and the requirement $X_k\perp\varepsilon_k$ from {\bf (A)} we obtain
\begin{align*}
\Delta(v_j) &= \E\ip{X_1}{v_j}_{H_1} Y_1= \E\ip{X_1}{v_j}_{H_1}\Psi(X_1) + \E\ip{X_1}{v_j}_{H_1}\varepsilon_1\\
&=\Psi(\E\ip{X_1}{v_j}_{H_1}X_1)= \Psi(C(v_j))= \lambda_j\Psi(v_j).
\end{align*}
Then, for any
$x \in H_1$, the derived equation leads to the representation
\begeq\label{rep}
\Psi(x) = \Psi\Bigg(\sum_{j=1}^\infty \ip{v_j}{x} v_j\Bigg) = \sum_{j=1}^\infty \frac{\Delta(v_j)}{\lambda_j} \ip{v_j}{x}.\label{psiniceform}
\eneq
Here we assume implicitly that $\mathrm{dim}(H_1)=\infty$. If $\mathrm{dim}(H_1)=M<\infty$, then \eqref{rep} still holds with $\infty$ replaced by $M$. This case is well understood and will therefore be excluded.

Equation \eqref{psiniceform} gives a core idea for estimation of $\Psi$. We will estimate $\Delta$, $v_j$ and $\lambda_j$ from our sample $X_1,\ldots,X_n,Y_1,\ldots,Y_n$ and substitute the estimators into formula \eqref{psiniceform}. The estimated eigenelements $(\hat\lambda_{j,n},\hat v_{j,n};\, 1\leq j\leq n)$ will be obtained from
the empirical covariance operator
\begin{equation*}
\hat C_n=\frac{1}{n}\sum_{k=1}^n X_k\otimes X_k.
\end{equation*}
In a similar straightforward manner we set
\begin{equation*}
\hat \Delta_n=\frac{1}{n}\sum_{k=1}^n X_k\otimes Y_k.
\end{equation*}
For ease of notation, we will suppress in the sequel the dependence on the sample size $n$ of these estimators.

Apparently, from the finite sample we cannot estimate the entire sequence $(\lambda_j,v_j)$, rather we have to work with a truncated version. This leads to
\begeq\label{psi_est}
\hat \Psi_{K}(x)=\sum_{j=1}^{K} \frac{\hat\Delta(\hat v_j)}{\hat \lambda_j} \ip{\hat v_j}{x},
\eneq 
where the choice of $K=K_n$ is crucial. Since we want our estimator to be consistent, $K_n$ has to grow with the sample size to infinity. On the other hand, we know that $\lambda_j\to 0$. Hence, it will be a delicate issue to control the behavior of $\frac{1}{\hat{\lambda}_j}$. A small error in the estimation of $\lambda_j$ can have an enormous impact on \eqref{psi_est}. 

Define $\Psi_K(x)=\sum_{j=1}^{K} \frac{\Delta(v_j)}{\lambda_j} \ip{v_j}{x}$. Via the ergodic theorem one can show that the individual terms $\hat\lambda_j$, $\hat v_j$ and $\hat\Delta$ in \eqref{psi_est} converge to their population counterparts. It follows that $\|\hat\Psi_K-\Psi_K\|_{\mathcal{L}_{12}}\to 0$ a.s., as long as $K$ is fixed. In fact, this holds true under finite second moments. However, as it is well known, the ergodic theorem doesn't assure rates of convergence. Even if the underlying random variables were bounded, convergence can be arbitrarily slow. Consequently, we cannot let $K$ grow with the sample size in this approach. We need to impose further structure on the dynamics of the process and existence of higher order moments. Both are combined in the concept of $L^4$--$m$--approximability.

In most existing papers determination of $K_n$ is related to the decay-rate of $\{\lambda_j\}$. For example, Cardot et al.~\cite{cardot:ferraty:sarda:1999} assume that $n\lambda_{K_n}^4\to\infty$ and $n\lambda_{K_n}^2 / (\sum_{j=1}^{K_n}\frac{1}{\alpha_j})^2\to \infty$, when 
\begin{equation}\label{alpha}
\alpha_1=\lambda_1-\lambda_2\quad \text{and}\quad\alpha_j=\min\{\lambda_{j-1}-\lambda_j,\lambda_j-\lambda_{j+1}\},\quad j>1.
\end{equation} 
Similar requirements are used in Bosq~\cite{b00} (Theorem 8.7) or Yao et al.~\cite{yao:muller:wang:2005} (Assumption (B.5)). 
Hall and Horowitz~\cite{hall:horowitz:2007} assume in the scalar response model that $\alpha_j\geq C^{-1}j^{-\alpha-1}$, $|\Delta(v_j)\lambda_j^{-1}|\leq C j^{-\beta}$ for some $\alpha>1$ and $\frac{1}{2}\alpha+1<\beta$. Here $C$ is a constant arising from the additional assumption $E\langle X_1,v_j\rangle^4\leq C\lambda_j^2$. They emphasize the importance of a sufficient separation of the eigenvalues for their result. Then, within this setup, optimal minimax bounds are proven to hold for $K=n^{1/(\alpha+2\beta)}$. Of course, in practice this choice of $K$ is only possible under the unrealistic assumption that we know $\alpha$ and $\beta$.
Cai and Zhou~\cite{cai:zhou:2008} modify the approach by Hall and Horowitz~\cite{hall:horowitz:2007} by proposing an adaptive choice of $K$ which is based on a block thresholding technique. They recover the optimal rates of Hall and Horowitz~\cite{hall:horowitz:2007}, but need to impose further technical assumptions. Among others,  the assumptions in \cite{hall:horowitz:2007} are strengthened to $E\|X_k\|^p<\infty$ for all $p>0$,   $j^{-\alpha}\ll\lambda_j\ll j^{-\alpha}$, and $\alpha_j\gg j^{-\alpha-1}$. Here $a_n\ll b_n$ means that $\limsup_n |a_n/b_n|<\infty$.  Rates of convergence are also obtained in Cardot and Johannes~\cite{cardot:johannes:2010}. They propose a new class of estimators which are based on projecting on some fixed orthonormal basis instead on empirical eigenfunctions. Again, the accuracy of the estimator relies on a thresholding technique, and similar as to the afore cited papers, the very strong results are at the price of several technical constraints.

\subsection{Consistency results}
\label{ss:results}

The papers cited in the previous paragraph are focus on rates of consistency for the estimator $\hat\psi_K$. These important and interesting need to impose technical assumptions on the operator $\Psi$ and the spectrum of $C$. In practice, such technical conditions cannot be checked and may be violated. Furthermore, since we have no knowledge of $\alpha_j$ and $\lambda_j$, $j\geq 1$, determination of $K$ has to be done heuristically. It then remains open if the widely used PC based estimation methods stay consistent in the case where some of these conditions are violated. Our theorems below show that the answer to this question is affirmative, even if data are dependent. We propose a selection of $K_n$ which is data driven and can thus be practically implemented. 
%
%
%
The $K_n$ we use in first result, Theorem~\ref{thm:main} below, is given as follows:\medskip

\noindent
{\bf (K):} {\em Let $m_n\to\infty$ such that $m_n^6=o(n)$. Then we define $K_n =\min(B_n,E_n,m_n)$
where $B_n =\arg\max \{j \geq 1 | \hat{\lambda}_j^{-1} \leq m_n \}$ and $E_n = \arg\max\{k \geq 1 |  \max_{1 \leq j \leq k}\hat{\alpha}_j^{-1} \leq m_n\}$.
Here $\hat\lambda_j$ and $\hat\alpha_j$ are the estimates for $\lambda_j$ and $\alpha_j$ (given in \eqref{alpha}), respectively, obtained from $\hat C$.}\medskip

A discussion on the tuning parameter $m_n$ is given at the end of this section.
The choice of $K_n$ is motivated by a `bias variance trade-off' argument. If an eigenvalue is very small (in our case $\ll 1/m_n$) it means that the direction it explains has only small influence on the representation of $X_k$. Therefore, excluding it from the representation of $\Psi$ will not cause a big bias, whereas it will considerably reduce the variance. It will be only included if the sample size is big enough, in which case we can hope for a reasonable accuracy of $\hat\lambda_j$. 
In practice it is recommended to replace $\frac{1}{\hat{\lambda}_j}$ in the definition of $B_n$ by $\frac{\hat\lambda_1}{\hat{\lambda}_j}$ and $\frac{1}{\hat{\alpha}_j}$ in the definition of $E_n$ by $\frac{\hat\lambda_1}{\hat{\alpha}_j}$ to adapt for scaling. For the asymptotics such a modification has no influence.

\begin{thm}\label{thm:main}
Consider the linear Hilbertian model \eqref{eq:model} and assume that Assumption~{\bf (A)} and {\bf (K)} hold. Suppose further that the eigenvalues $\{\lambda_j\}$  are mutually distinct and $\Psi$ is a Hilbert-Schmidt operator. Then the estimator described in Section~\ref{ss:estimator} is weakly consistent, i.e.
$
\|\hat{\Psi}_{K_n} - \Psi\|_{\cL_{12}}\convP 0, 
$
if $n\to \infty$.
\end{thm}

It is not hard to see that consistent estimation of $\Psi$ via the PCA approach requires compactness of the operator. As a simple example suppose that $\Psi$ is the identity operator, which is not Hilbert-Schmidt anymore. Then for any ONB $\{v_i\}$ we have $\Psi=\sum_{i\geq 1}v_i\otimes v_i$. Even if from the finite sample our estimators for $v_1,\ldots, v_K$ would be perfect ($v_i=\hat v_i$) we have $\|\Psi-\hat \Psi_K\|_{\cL_{12}}=1$ for any $K\geq 1$. This is easily seen by evaluating $\Psi$ and $\hat \Psi_K$ at $v_{K+1}$.

In our next theorem we show that if our target is prediction, then we can further simplify the assumptions. In this case we will be satisfied if  $\|\Psi(X_n)-\hat\Psi(X_n)\|_{H_2}$ is small. E.g., if $\langle X_n,v\rangle=0$ with probability one, then the direction $v$ plays no role for describing $X_n$ and a larger value of $\|\Psi(v)-\hat\Psi(v)\|_{H_2}$ is not relevant.

\begin{thm}\label{thm:convX} Let Assumption {\bf (A)} hold and define the estimator $\hat\Psi_{K_n}$ as in Section~\ref{ss:estimator} with $K_n = \arg\max \{j \geq 1 | \ \hat\lambda_1/\hat{\lambda}_j \leq m_n \}$, where $m_n\to\infty$ and $m_n = o(\sqrt{n})$. Then  
$
\|\Psi(X_n) - \hat\Psi_{K_n}(X_n)\|_{H_2} \xrightarrow{\Prob} 0.
$
\end{thm}
\begin{rmk}
For our proof it will not be important to evaluate $\Psi$ and $\hat\Psi$ at $X_n$. We could equally well use $X_1$, or $X_{n+1}$, or some arbitrary variable $X\isd X_1$.
\end{rmk}

Theorem~\ref{thm:convX} should be compared to Theorem~3 in Crambes and Mas~\cite{crambes:mas:2012} where an asymptotic expansion of 
$
E\|\Psi(X_{n+1}) - \hat\Psi_{k}(X_{n+1})\|_{H_2}^2 
$
is obtained (for fixed $k$). Their result implies consistency, but requires again assumptions on the decay rate of $\{\lambda_i\}$, an operator $\Psi$ that is Hilbert-Schmidt, and $E\|X_k\|^p<\infty$ for all $p>0$. In our theorem we need no assumptions on the eigenvalues anymore, not even that they are distinct.


In the last theorem we saw that whenever $m_n=o(\sqrt{n})$ and $m_n\to\infty$ convergence holds. This leaves open what is a good choice of the tuning parameter $m_n$. From a practical perspective we believe that the importance of this question should not be overrated. Most applied researchers will use CV or some comparable method, which usually will give a $K_n^{\mathrm{alt}}$ that is presumably close to optimal. Hence, if we suppose that
$$
E \|\Psi(X_n) - \hat\Psi_{K_n^{\mathrm{alt}}}(X_n)\|_{H_2}\ll E\|\Psi(X_n) - \hat\Psi_{K_n}(X_n)\|_{H_2}\quad (n\to\infty),
$$
the practitioner can be sure that his approach leads to a consistent estimator under very general assumptions. In Section~\ref{s:simulation} we use for the simulations $m_n=\sqrt{n}/\log n$. The performance of this estimator is in all tested setups comparable to CV.

To address the optimality issue from a theoretical point of view seems to be very difficult and depends on our final objective: is it prediction or estimation.
In both cases we believe that results in this direction can only be realistically obtained under regularity assumptions similar to those in the above cited articles.

\subsection{Applications to functional time series}

Functional time series analysis has seen an upsurge in FDA literature, in particular the forecasting in a functional setup (see e.g.\ Hyndman and Shang~\cite{hyndman:shang:2009} or Sen and Kl\"uppelberg~\cite{sen:klueppelberg:2010}). We sketch here two possible applications in this context.

\subsubsection{FAR(1)}
Of particular importance in functional time series is the ARH(1) model of Bosq~\cite{b00}.
We show now
that our framework covers this model. With i.i.d.\ innovations $\delta_k\in L_H^4$ the process $\{X_k\}$ defined via $X_{k+1}=\Psi(X_k)+\delta_{k+1}$  is $L_H^4$--approximable if $\Psi\in \cL(H,H)$ such that $\|\Psi\|_{\cL(H,H)}<1$, see \cite{hk10}. The stationary solution for $X_k$ has the form
$$
X_k=\sum_{j\geq 0}\Psi^j(\delta_{k-j}).
$$
Setting $\varepsilon_{k}=\delta_{k+1}$ and $Y_k=X_{k+1}$ we obtain the linear model \eqref{eq:model}.  Independence of $\{\delta_k\}$ implies that $X_k\perp \varepsilon_k$ and hence
Assumption {\bf (A)} holds. Bosq~\cite{b00} has obtained a (strongly) consistent estimator of $\Psi$, if $\Psi$ is Hilbert-Schmidt and again by imposing assumptions on the spectrum of $C$.

In our approach we don't even need that the innovations $\{\delta_k\}$ are i.i.d. As long as we can assure that $\{\delta_k\}$ and $\{X_k\}$ are $L^4$--$m$--approximable we only need that
$\{\delta_k\}$ is $H$-white noise. Indeed, denoting $A^*$ the conjugate of operator $A$, we have for any $x\in H_1$ and $y\in H_2$ that
\begin{align*}
E\langle X_k,x\rangle_{H_1}\langle\varepsilon_k,y\rangle_{H_2}&=\sum_{j\geq 0}E\langle\Psi^j(\delta_{k-j}),x\rangle_{H_1}\langle\delta_{k+1},y\rangle_{H_2}\\
&=\sum_{j\geq 0}E\langle\delta_{k-j},(\Psi^j)^*(x)\rangle_{H_1}\langle\delta_{k+1},y\rangle_{H_2}=0.
\end{align*}
This shows $X_k\perp\varepsilon_k$ and Assumption~{\bf (A)} follows.

We obtain the following
\begin{cor}
Let $\{X_n\}_{n \geq 1}$ be an ARH(1) process given by the recurrence equation $X_{n+1}=\Psi(X_n)+\varepsilon_{n+1}$. 
Assume $\|\Psi\|_{\cL_{12}} < 1$. If $\{\varepsilon_i\}$ is $H$-white noise and Assumption {\bf (A)} holds, then for the estimator $\hat{\Psi}_K$ given in Theorem~\ref{thm:convX} we have $\|\Psi(X_n)-\hat\Psi_K(X_n)\|_{H_2}\convP 0$. In particular if $\{\varepsilon_i\}$ is i.i.d.\ in $L_{H}^4$, Assumption {\bf (A)} will hold.
\end{cor}

\begin{cor}
Let $\{X_n\}_{n \geq 1}$ be an ARH(1) process given by the recurrence equation $X_{n+1}=\Psi(X_n)+\varepsilon_{n+1}$. 
Assume $\|\Psi\|_{\cS_{12}} < 1$  and that the covariance operator related to $X_1$ has distinct eigenvalues. If $\{\varepsilon_i\}$ is $H$-white noise and {\bf (A)} and {\bf (K)} hold, then the estimator $\hat{\Psi}_K$ is consistent.
\end{cor}

We remark that employing the usual state-space representation for FAR($p$) processes these results are easily generalized to higher order FAR models.

\subsubsection{FARCH(1)}

Another possible application of our result refers to a recently introduced functional version of the celebrated ARCH model (H\"ormann et al.~\cite{hormann:horvath:reeder:2012}), which plays a fundamental role in financial econometrics. It is given by the two equations
$$
y_k(t)=\varepsilon_k(t)\sigma_k(t),\quad t\in[0,1],\,k\in\mathbb{Z}
$$ 
and
$$
\sigma^2_k(t)=\delta(t)+\int_0^1\beta(t,s)y_{k-1}^2(s)ds,\quad t\in[0,1],\,k\in\mathbb{Z}.
$$
Without going into details, let us just mention that one can write the squared observations of a functional ARCH model as an autoregressive process with innovations $\nu_k(t)=y_k^2(t)-\sigma^2_k(t)$. The new noise $\{\nu_k\}$ is no longer independent and hence the results of \cite{b00} are not applicable to prove consistency of the involved estimator for the operator $\beta$. But it is shown in \cite{hormann:horvath:reeder:2012} that the innovations of this new process form Hilbertian white noise and that the new process is $L^4$--$m$--approximable. This allows us to obtain a consistent estimator for $\beta$.\medskip

\section{Simulation study}\label{s:simulation}
We consider a linear model of the form $Y_n=\Psi(X_n)+\varepsilon_n$, where $X_1,\varepsilon_1,X_2,\varepsilon_2,\ldots$ are mutually independent. 
 We are testing the performance of the estimator in context of prediction, i.e.\ we work under the setting of Theorem~\ref{thm:convX}. For the simulation study we obviously have to work with finite dimensional spaces $H_1$ and $H_2$. However, because of the asymptotic nature of our results, we set the dimension relatively high and define $H_1 = H_2 = \mathrm{span}\{f_j : 0 \leq j \leq 34\}$, where $f_0(t)=1$, $f_{2k-1}(t)=\sin(2\pi kt)$ and $f_{2k}(t)=\cos(2\pi kt)$ are the first 35 elements of a Fourier basis on $[0,1]$.
We work with Gaussian curves $X_i(t)$ by setting 
\begin{align}
X_i (t)= \sum_{j=0}^{34} A_i^{(j)} f_{j-1}(t)\label{eq:Xrepr},
\end{align}
where $(A_i^{(0)},A_i^{(1)},\ldots,A_{i}^{(34)})'$ are independent Gaussian random vectors with mean zero and covariance $\Sigma$.
This setup allows us to easily manipulate the eigenvalues $\{\lambda_k\}$ of a covariance operator $C_X=EX\otimes X$. Indeed, if we define $\Sigma=\mathrm{diag}(a_1,\ldots,a_{35})$, where $a_1\geq a_2\geq\cdots\geq a_k$,
then $\lambda_k=a_k$ and $v_k=f_{k-1}$ is the corresponding eigenfunction. We test three sets of eigenvalues $\{\lambda_k\}_{1 \leq k\leq 35}$:
\begin{itemize}
\item $\Lambda_1:$ $\lambda_k = c_1 \rho^{k-1}$ with $\rho=1/2$;\quad[geometric decay],
\item $\Lambda_2:$ $\lambda_k = c_2/k^2$ \quad[fast polynomial decay],
\item $\Lambda_3:$ $\lambda_k = c_3/k^{1.1}$ \quad[slow polynomial decay].
\end{itemize}
To bring our data on the same scale and make results under different settings comparable we set $c_1,c_2$ and $c_3$ such that $\sum_{k=1}^{35}\lambda_k=1$. This implies $E\|X_i\|^2=1$ in all settings.
The noise $\{\varepsilon_k\}$ is also assumed to be of the form \eqref{eq:Xrepr}, but now with $E\|\varepsilon_i\|^2= \sigma^{2} \in \{0.25,1,2.25,4\}$.

We test three operators, all of the form
$\Psi(x)=\sum_{i=1}^{35}\sum_{j=1}^{35}\psi_{ij}\langle x,v_i\rangle v_j$.
\begin{itemize}
\item $\Psi_1:$ for $1 \leq i,j \leq 35$ we set $\psi_{ii} = 1$ and $\psi_{ij} = 0$ when $i \neq j$,
\item $\Psi_2:$ the coefficients $\psi_{ij}$ are generated as i.i.d.\ standard normal random variables,
\item $\Psi_3:$ for $1 \leq i,j \leq 35$ we set $\psi_{ij} = \frac{1}{ij}$
\end{itemize}
We standardize the operators such that the operator norm equals one. The operators $\Psi_2$ are generated once and then fixed for the entire simulation. We generate samples of size $n+1=80\times 4^\ell$+1, $\ell=0,\ldots,4$. Estimation is based on the first $n$ observations. We run $200$ simulations for each setup $(\Lambda,\Psi,\sigma, n)$. As a performance measure for our procedure the mean squared error on the $(n+1)$-st observation 
\begin{align}
{\tt MSE}  = \frac{1}{200}\sum_{k=1}^{200} \|\Psi(X^{(k)}_{n+1}) - \hat \Psi(X^{(k)}_{n+1})\|^2_{H_2},\label{eq:MSE}
\end{align}
is used. Here $X_i^{(k)}$ is the $i$-th observation of the $k$-th simulation run. 

Now we compute the median truncation level $K$ obtained from our data-driven procedure described in Theorem~\ref{thm:convX} with $m_n=\frac{n^{1/2}}{\log n}$.
We compare it to the median truncation level obtained by cross-validation ($K^{CV}$) on the same data. To this end, we divide the sample into training and test sets in proportion $(n-n_{\mathrm{test}}):n_{\mathrm{test}}$, where $n_{\mathrm{test}}=\max\{n/10,100\}$. The estimator is obtained from the training set for different truncation levels $k=1,2,\ldots,35$. Then, from the test set we determine $K^{CV}=\mathrm{argmin}_{k\in\{1,\ldots,35\}}\sum_{\ell=n-n_{\mathrm{test}}}^n\|Y_{\ell+1}-\hat\Psi_k(X_\ell)\|_{H_2}^2$.

The {\tt MSE} and the size of $K$ and $K^{CV}$ are shown for different constellations in Table~\ref{tb:res}. We display the results only for $\sigma=1$. Not surprisingly, the bigger the variance of the noise, the bigger {\tt MSE}, but otherwise our findings were the same across all constellations of $\sigma$. The table shows that the choice of $K$ proposed by our method results in an MSE which is competitive with CV. We also see that an optimal choice of $K$ cannot be solely based on the decay of the eigenvalues as it is the case in our approach. It clearly also depends on the unknown operator itself. Not surprisingly, the best results are obtained under settings $\Lambda_1$ (exponentially fast decay of eigenvalues) and $\Psi_3$ (which is the smoothest among the three operators).

\begin{table}[h!]\label{tb:res}
\caption{Truncation levels obtained by Theorem~\ref{thm:convX} ($K$) and by cross-validation ($K^{CV}$) and corresponding {\tt MSE}. For each constellation we present $\mathrm{med}(K)$ of $200$ runs.}
\begin{center}
{\small \begin{tabular}{rr|rrrr|rrrr|rrrr}
& & \multicolumn{4}{c|}{$\Psi_1$} & \multicolumn{4}{c}{$\Psi_2$} & \multicolumn{4}{c}{$\Psi_3$}\\
& $n$ & $K^{CV}$ & MSE & $K$ & MSE & $K^{CV}$ & MSE & $K$ & MSE & $K^{CV}$ & MSE & $K$ & MSE\\
\hline
\multirow{5}{*}{$\Lambda_1$} & 80 & 1 & 1.10 & 2 & 0.96 & 1 & 0.68 & 2 & 0.69 & 1 & 0.64 & 2 & 0.66\\
 & 320 & 3 & 0.48 & 2 & 0.43 & 1 & 0.32 & 2 & 0.28 & 1 & 0.21 & 2 & 0.24\\
 & 1280 & 4 & 0.21 & 3 & 0.21 & 3 & 0.14 & 3 & 0.12 & 2 & 0.09 & 3 & 0.09\\
 & 5120 & 7 & 0.08 & 4 & 0.10 & 5 & 0.07 & 4 & 0.05 & 3 & 0.05 & 4 & 0.03\\
 & 20480 & 9 & 0.03 & 4 & 0.06 & 8 & 0.03 & 4 & 0.02 & 5 & 0.02 & 4 & 0.01\\
\hline
\multirow{5}{*}{$\Lambda_2$} & 80 & 1 & 1.00 & 1 & 0.85 & 1 & 0.82 & 1 & 0.58 & 1 & 0.56 & 1 & 0.4\\
 & 320 & 2 & 0.56 & 1 & 0.54 & 1 & 0.26 & 1 & 0.22 & 1 & 0.20 & 1 & 0.15\\
 & 1280 & 5 & 0.26 & 2 & 0.28 & 2 & 0.14 & 2 & 0.12 & 1 & 0.07 & 2 & 0.06\\
 & 5120 & 9 & 0.13 & 2 & 0.24 & 5 & 0.08 & 2 & 0.08 & 3 & 0.04 & 2 & 0.02\\
 & 20480 & 17 & 0.06 & 3 & 0.16 & 10 & 0.04 & 3 & 0.04 & 5.5 & 0.02 & 3 & 0.01\\
\hline
\multirow{5}{*}{$\Lambda_3$} & 80 & 1 & 1.60 & 2 & 1.30 & 1 & 0.78 & 1 & 0.73 & 1 & 0.71 & 1 & 0.57\\
 & 320 & 2 & 0.85 & 2 & 0.78 & 1 & 0.35 & 2 & 0.40 & 1 & 0.22 & 2 & 0.28\\
 & 1280 & 8 & 0.55 & 4 & 0.55 & 2 & 0.22 & 4 & 0.22 & 2 & 0.08 & 4 & 0.12\\
 & 5120 & 24 & 0.25 & 6 & 0.38 & 9 & 0.16 & 6 & 0.14 & 3 & 0.04 & 6 & 0.04\\
 & 20480 & 33 & 0.08 & 11 & 0.25 & 23 & 0.07 & 11 & 0.08 & 5 & 0.02 & 11 & 0.02
 
\end{tabular}}
\end{center}
\end{table}

\section{Conclusion}\label{s:conclusion}

Estimation of the regression operator in functional linear models has obtained much interest over the last years. Our objective in this paper was to show that one of most widely applied estimators in this context remains consistent, even if several of the synthetic assumptions used in previous papers are removed. If our intention is prediction, we can further simplify the technical requirements.  Our approach comes with a data driven choice of the parameter which determines the dimension of the estimator. While our main intention is to show that this choice leads to a consistent estimator, we have seen in simulations that our method is performing remarkably well when compared to cross-validation.

\section{Proofs}\label{s:proofs}

Throughout this entire section we assume the setup and notation of Section~\ref{s:setup}.

\subsection{Proof of Theorem~\ref{thm:main}}

We work under Assumptions {\bf (A)} and {\bf (K)} and assume distinct eigenvalues of the covariance operator $C$ and that $\Psi$ is Hilbert-Schmidt.
The first important lemma which we use in the proof of Theorem~\ref{thm:main} is an error bound for the estimators of the operators $\Delta$ and $C$. Below we extend results in~\cite{hk10}.
\begin{lem}\label{Lemma1}
There is a constant $U$ depending only on the law of $\{(X_k,Y_k)\}$ such that
\begin{align*}
n\max\{
\E\|\Delta - \hat{\Delta}_n\|_{\mathcal{S}_{12}}^2,\E\|C - \hat{C}_n\|_{\mathcal{S}_{11}}^2
\} < U.
\end{align*}
\end{lem}
\begin{proof}[Proof of Lemma \ref{Lemma1}] We only prove the bound for $\Delta$, the one for $C$ is similar. First note that by Lemma~2.1 in \cite{hk10} and Assumption~{\bf (A)} $\{Y_k\}$ is also $L^4$--$m$--approximable. Next we observe that
\begeqO
n\E\big\|\Delta - \hat{\Delta}_n\big\|_{\cS_{12}}^2 = n\E\left\|\frac{1}{n}\sum\limits_{k=1}^nZ_k\right\|_{\cS_{12}}^2,
\eneqO
where $Z_k = X_k \otimes Y_k - \Delta$. Set $Z_k^{(r)} = X_k^{(r)} \otimes Y_k^{(r)} - \Delta$. Using the stationarity of the sequence $\{Z_k\}$ we obtain
\begin{align}
n\E \left\|\frac{1}{n}\sum\limits_{k=1}^n Z_k\right\|_{\cS_{12}}^2 &= \sum\limits_{|r|<n}\left(1- \frac{|r|}{n}\right)\E\ip{Z_0}{Z_r}_{\cS_{12}}\nonumber\\
&\leq \E\|Z_0\|_{\cS_{12}}^2 + 2\sum\limits_{r=1}^\infty |\E\ip{Z_0}{Z_r}_{\cS_{12}}|.\label{sum}
\end{align}
By the Cauchy-Schwarz inequality and the independence of $Z_r^{(r-1)}$ and $Z_0$ we derive:
\begin{align*}
|\E\ip{Z_0}{Z_r}_{\cS_{12}}|=|\E\ip{Z_0}{Z_r-Z_r^{(r-1)}}_{\cS_{12}}| \leq (\E\|Z_0\|_{\cS_{12}}^2)^{\frac{1}{2}} (\E\|Z_r - Z_r^{(r-1)}\|_{\cS_{12}}^2)^{\frac{1}{2}}.
\end{align*}
Using $\|X_0\otimes Y_0\|_{\cS_{12}}= \|X_0\|_{H_1} \|Y_0\|_{H_2}$ and again the Cauchy-Schwarz inequality we get
\begin{align*}
\E\|Z_0\|_{\cS_{12}}^{2} =  \E\|X_0\|_{H_1}^{2} \|Y_0\|_{H_2}^{2} \leq \nu_{4,H_1}^2(X_0)\nu_{4,H_2}^2(Y_0)<\infty.
\end{align*}

To finish the proof we show that $\sum\limits_{r=1}^\infty (\E\|Z_r - Z_r^{(r-1)}\|_{\cS_{12}}^2)^{\frac{1}{2}}<\infty$. By using an inequality of the type $|ab-cd|^{2} \leq 2|a|^{2}|b-d|^{2} + 2|d|^{2}|a-c|^{2}$ we obtain
\begin{align*}
\E\|Z_r - Z_r^{(r-1)}\|_{\cS_{12}}^2 &= \|X_r\otimes Y_r - X_r^{(r-1)} \otimes Y_r^{(r-1)}\|_{\cS_{12}}^2\\
&\leq 2\E \|X_r\|^2_{H_1} \|Y_r - Y_r^{(r-1)}\|^2_{H_2} + 2\E\|Y_r^{(r-1)}\|^2_{H_2} \|X_r - X_r^{(r-1)}\|^2_{H_1}\\
&\leq 2 \nu_{4,H_1}^2(X_r) \nu_{4,H_2}^2(Y_r - Y_r^{(r-1)}) + 2 \nu_{4,H_2}^2(Y_r^{(r-1)}) \nu_{4,H_1}^2(X_r - X_r^{(r-1)}).
\end{align*}
Convergence of \eqref{sum} follows now directly from $L^4$-$m$--approximability.
\end{proof}

Application of this lemma leads also to bounds for estimators of eigenvalues and eigenfunctions of $C$ via the following two lemmas (see \cite{hk10}).

\begin{lem}\label{Lemma2}
Suppose $\lambda_i, \hat{\lambda}_i$ are the eigenvalues of $C$ and $\hat C$, respectively, listed in decreasing order. Let $v_i, \hat{v}_i$ be the corresponding eigenvectors and let $\hat c_i = \ip{v_i}{\hat{v}_i}$. Then for each $j\geq 1$, 
\begin{align*}
\hat{\alpha}_j\|v_j - \hat{c}_j\hat{v}_j\|_{H_1} \leq 2\sqrt{2} \|\hat{C} - C\|_{\cL_{11}},
\end{align*}
where $\hat{\alpha}_j=\min\{\hat{\lambda}_{j-1}-\hat{\lambda}_{j},\hat{\lambda}_{j}-\hat{\lambda}_{j+1}\}$ and $\hat\alpha_1=\hat\lambda_2-\hat\lambda_1$.
\end{lem}

\begin{lem}\label{Lemma3}
Let $\lambda_j, \hat{\lambda}_j$ be defined as in Lemma \ref{Lemma2}. Then for each $j\geq 1$, 
\begin{align*}
|\lambda_j - \hat{\lambda}_j| \leq \|C - \hat{C}\|_{\cL_{11}}.
\end{align*}
\end{lem}

In the following calculations we work with finite sums of the representation in \eqref{psiniceform}:
\begeq\label{eq:psik}
\Psi_K(x) = \sum_{j=1}^K \frac{\Delta(v_j)}{\lambda_j} \ip{v_j}{x}.
\eneq
In order to prove the main result we consider the term $\|\Psi-\hat{\Psi}_K\|_{{\cL_{12}}}$ and decompose it using the triangle inequality into four terms
\begin{align*}
\|\Psi - \hat{\Psi}_K\|_{\cL_{12}} \leq \sum_{i=1}^4\|S_i(K)\|_{\cL_{12}},
\end{align*}
where
\begin{align}
S_1(K) &= \sum_{j=1}^K\left( \hat{c}_j\hat{v}_j \otimes \frac{\hat{\Delta}(\hat{c}_j\hat{v}_j)}{\hat{\lambda}_j} - \hat{c}_j\hat{v}_j \otimes \frac{{\Delta}(\hat{c}_j\hat{v}_j)}{\hat{\lambda}_j}\right), \label{dc:t1}\\
S_2(K) &= \sum_{j=1}^K\left( \hat{c}_j\hat{v}_j \otimes \frac{{\Delta}(\hat{c}_j\hat{v}_j)}{\hat{\lambda}_j} - \hat{c}_j\hat{v}_j \otimes \frac{{\Delta}(\hat{c}_j\hat{v}_j)}{\lambda_j}\right),\label{dc:t2}\\
S_3(K) &= \sum_{j=1}^K \left(\hat{c}_j\hat{v}_j \otimes \frac{{\Delta}(\hat{c}_j\hat{v}_j)}{{\lambda_j}} - v_j \otimes \frac{{\Delta}(v_j)}{\lambda_j}\right),\label{dc:t3}\\
S_4(K) &= \Psi - \Psi_K.\label{dc:t4}
\end{align}
The following simple lemma gives convergence of $S_4(K_n)$, provided $K_n\convP\infty$.

\begin{lem}\label{lem:r1}
Let $\{{K_n},n\geq 1\}$ be a random sequence taking values in $\mathbb{N}$, such that ${K_n}\convP\infty$ as $n\to\infty$. Then $\Psi_{{K_n}}$ defined by the equation \eqref{eq:psik} converges to $\Psi$ in probability.
\end{lem}}
\begin{proof}
Notice that since $\|\Psi\|^2_{\cS_{12}} = \sum\limits_{j=1}^\infty \|\Psi(v_j)\|_{H_2}^2 < \infty$ for some orthonormal base $\{v_j\}$, we can find $m_\varepsilon \in \mathbb{N}$ such that $\|\Psi-\Psi_m\|_{\cS_{12}}^2=\sum\limits_{j>m} \|\Psi(v_j)\|_{H_2}^2 \leq \varepsilon$, whenever $m>m_\varepsilon$. Hence
\begin{align*}
\Prob(\|\Psi - \Psi_{{K_n}}\|^{2}_{\cS_{12}} > \varepsilon) &= \sum_{m=1}^\infty \Prob(\|\Psi - \Psi_{m}\|_{\cS_{12}}^{2} > \varepsilon \cap {K_n} = m)\\
&= \Prob({K_n}\leq m_\varepsilon).
\end{align*}
\end{proof}

The next three lemmas deal with terms \eqref{dc:t1}--\eqref{dc:t3}. 
\begin{lem} \label{lem:d1}
Let $S_1(K)$ be defined by the equation \eqref{dc:t1} and $U$ the constant derived in Lemma \ref{Lemma1}. Then
\begin{align*}
\Prob ( \|S_1({K_n})\|_{\cL_{12}} > \varepsilon ) \leq \frac{U m_n^2}{\varepsilon^2 n}.
\end{align*}
\end{lem}
\begin{proof}
Note that for an orthonormal system $\{e_i \in H_1\ |\ i \geq 1\}$ and any sequence $\{x_i \in H_2\ |\ i \geq 1\}$  the following identity holds:
\begin{align}
\Bigg\|\sum_{i=1}^{K} e_i \otimes x_i\Bigg\|_{\cS_{12}}^2 = \sum_{j=1}^{\infty}\Bigg\|\sum_{i=1}^{K} \ip{e_i}{e_j} x_i\Bigg\|_{H_2}^2 = \sum_{j=1}^K\|x_j\|_{H_2}^2\label{eq:hs}.
\end{align}
Using this and the fact that the Hilbert-Schmidt norm bounds the operator norm we derive
\begin{align*}
\Prob ( \|S_1({K_n})\|^2_{\cL_{12}} > \varepsilon )  &\leq \Prob\Bigg(\Bigg\| \sum\limits_{j=1}^{{K_n}} \hat{c}_j\hat{v}_j \otimes \frac{1}{\hat{\lambda}_j}(\hat{\Delta} - \Delta)(\hat{c}_j\hat{v}_j) \Bigg\|^2_{{\cS_{12}}} > \varepsilon\Bigg) \\
&\leq \Prob\Bigg(\frac{1}{\hat{\lambda}_{K_n}^{2}}\sum\limits_{j=1}^{{K_n}} \| (\hat{\Delta} - \Delta)(\hat{c}_j\hat{v}_j) \|^2_{{H_2}} > \varepsilon\Bigg) \\
&\leq \Prob( m_n^{2} \|\hat{\Delta} - \Delta \|^2_{{\cS_{12}}} > \varepsilon).
\end{align*}
By the Markov inequality
\begin{align*}
\Prob ( \|S_1({K_n})\|^2_{\cL_{12}} > \varepsilon ) \leq \E\|\hat{\Delta}-\Delta\|_{\cS_{12}}^2 \frac{m_n^2}{\varepsilon} \leq  U \frac{m_n^2}{\varepsilon n},
\end{align*}
where the last inequality is obtained from Lemma $\ref{Lemma1}$.
\end{proof}
\begin{lem} \label{lem:d2}
Let $S_2(K)$ be defined by the equation \eqref{dc:t2} and $U$ the constant from Lemma \ref{lem:d1}. Then 
\begin{align*}
\Prob ( \|S_2({K_n})\|_{\cL_{12}} > \varepsilon ) \leq 4U\|\Delta\|^2_{\cS_{12}}\frac{m_n^4}{\varepsilon^2 n}.
\end{align*}
\end{lem}

\begin{proof}
Assumption ${K_n} \leq B_n$ and identity (\ref{eq:hs}) imply
\begin{align*}
&\Prob ( \|S_2({K_n})\|^2_{\cL_{12}} > \varepsilon )
= \Prob\Bigg( \Bigg\| \sum_{j=1}^{{K_n}} \Bigg( \frac{1}{\lambda_j} - \frac{1}{\hat{\lambda}_j} \Bigg) \hat{c}_j\hat{v}_j \otimes \Delta(\hat{c}_j\hat{v}_j) \Bigg\|^2_{\cL_{12}} > \varepsilon \Bigg)\\
&\quad\leq \Prob\Bigg( \max_{1\leq j\leq K_n} \Bigg(\frac{\hat{\lambda}_j - \lambda_j}{\hat{\lambda}_j\lambda_j}\Bigg)^2 \sum_{j=1}^{K_n}\|\Delta(\hat{c}_j\hat{v}_j)\|^2_{H_2} > \varepsilon \Bigg)\\
&\quad\leq \Prob\Bigg( \max_{1\leq j\leq K_n} \Bigg(\frac{\hat{\lambda}_j - \lambda_j}{\lambda_j}\Bigg)^2 > \frac{\varepsilon}{m_n^2 \|\Delta\|^2_{\cS_{12}}} \Bigg).
\end{align*}
For simplifying the notation let  $b^2 = \frac{\varepsilon}{m_n^2 \|\Delta\|^2_{\cS_{12}}}$, then
\begin{align*}
&\Prob ( \|S_2({K_n})\|^2_{\cL_{12}} > \varepsilon )
\leq \Prob\Bigg( \max_{1\leq j\leq K_n} \Bigg|\frac{\hat{\lambda}_j - \lambda_j}{\lambda_j}\Bigg| > b \Bigg)\\
&\quad\leq \Prob\Bigg( \frac{1}{{\lambda}_{K_n}} \max_{1\leq j\leq K_n} |\hat{\lambda}_j - \lambda_j| > b\ \cap\ \max_{1\leq j\leq K_n} |\hat{\lambda}_j - \lambda_j| \leq \frac{b}{2m_n} \Bigg) + \Prob\Big(\max_{1\leq j\leq K_n}|\hat{\lambda}_j - \lambda_j| > \frac{b}{2m_n}\Big).
\end{align*}
The first summand vanishes because 
\begin{align*}
&\Prob\Bigg( \frac{1}{{\lambda}_{K_n}} \max_{1\leq j\leq K_n} |\hat{\lambda}_j - \lambda_j| > b\ \cap\ \max_{1\leq j\leq K_n} |\hat{\lambda}_j - \lambda_j| \leq \frac{b}{2m_n} \Bigg) \\
&\quad\leq \Prob\Bigg( \frac{b}{{2\lambda}_{K_n}  m_n} > b\ \cap\ |\hat{\lambda}_{K_n} - \lambda_{K_n}| \leq \frac{b}{2m_n} \Bigg)\\
&\quad\leq \Prob\Bigg( \frac{1}{2m_n} > {\lambda}_{K_n}\ \cap\ |\hat{\lambda}_{K_n} - \lambda_{K_n}| \leq \frac{\sqrt{\varepsilon}}{m_n^2 2\|\Delta\|_{\cS_{12}^2}} \Bigg),
\end{align*}
which is equal to $0$ for $n$ large enough, since $\hat{\lambda}_{K_n}\geq \frac{1}{m_n}$ and the distance between $\lambda_{K_n}$ and $\hat{\lambda}_{K_n}$ shrinks faster than $\frac{1}{2m_n}$. For the second term we use Lemma \ref{Lemma3} and the Markov inequality:
\begin{align*}
\Prob ( \|S_2({K_n})\|^2_{\cL_{12}} > \varepsilon )
&\leq \Prob\Big(\max_{1\leq j\leq K_n}|\hat{\lambda}_j - \lambda_j| > \frac{b}{2m_n} \Big)\\
&\leq \Prob\Big(\|\hat{C} - C\|_{\cL_{11}} > \frac{b}{2m_n} \Big)\\
&\leq \frac{4m_n^2 }{b^2} \E\|\hat{C} - C\|_{\cL_{11}}^2\\
&\leq 4U\|\Delta\|^2_{\cS_{12}}\frac{m_n^4}{\varepsilon n}.
\end{align*}
\end{proof}
\begin{lem} \label{lem:d3}
Let $S_3(K)$ be defined by \eqref{dc:t3} and $U$ be the constant defined in Lemma~\ref{lem:d1}, then
\begin{align*}
\Prob( \|S_3({K_n})\|_{\cL_{12}} < \varepsilon) \leq U(128\|\Delta\|_{\cL_{12}}^2 + 4\varepsilon^2) \frac{m_n^6}{\varepsilon^2 n}.
\end{align*}
\end{lem}

\begin{proof}
By adding and subtracting the term $\hat{c}_j \hat{v}_j \Delta(v_j)$ and using the triangle inequality we derive
\begin{align*}
\Prob ( \|S_3({K_n})\|_{\cL_{12}} > \varepsilon ) 
&= \Prob \Bigg( \Bigg\| \sum\limits_{j=1}^{{K_n}} \frac{1}{\lambda_j} (\hat{c}_j\hat{v}_j \otimes \Delta (\hat{c}_j\hat{v}_j) - v_j \otimes \Delta (v_j)) \Bigg\|_{\cL_{12}} > \varepsilon \Bigg)\\
&\leq \Prob \Bigg(  \sum\limits_{j=1}^{{K_n}} \frac{1}{\lambda_j} \| \hat{c}_j\hat{v}_j \otimes \Delta (\hat{c}_j\hat{v}_j - v_j) + (\hat{c}_j\hat{v}_j - v_j) \otimes \Delta (v_j) \|_{\cL_{12}} > \varepsilon \Bigg)\\
&\leq \Prob \Bigg( \sum\limits_{j=1}^{{K_n}} \frac{1}{\lambda_j} (\|\Delta \|_{\cL_{12}} \|\hat{c}_j\hat{v}_j - v_j\|_{H_1} + \|\hat{c}_j\hat{v}_j - v_j\|_{H_1} \| \Delta \|_{\cL_{12}}) > \varepsilon \Bigg).
\end{align*}
Now we split $\Omega=A\cup A^c$ where $A=\{\frac{1}{\lambda_{K_n}} > 2m_n\}$ and get
\begin{align}
\Prob ( \|S_3({K_n})\|_{\cL_{12}} > \varepsilon ) &\leq \Prob \Bigg( \frac{1}{\lambda_{{K_n}}} \sum\limits_{j=1}^{{K_n}} \|\hat{c}_j\hat{v}_j - v_j\|_{H_1} > \frac{\varepsilon}{2\|\Delta\|_{\cL_{12}}} \Bigg)\nonumber\\
&\leq \Prob \Bigg( \sum\limits_{j=1}^{{K_n}} \|\hat{c}_j\hat{v}_j - v_j\|_{H_1} > \frac{\varepsilon}{4m_n\|\Delta\|_{\cL_{12}}} \Bigg) + \Prob \Bigg(\frac{1}{\lambda_{K_n}} > 2m_n\Bigg).\label{eq:l6}
\end{align}
For the first term in the inequality (\ref{eq:l6}), by Lemma $\ref{Lemma2}$, definition of $E_n$ and the Markov inequality we get
\begin{align*}
\Prob \Bigg( \sum\limits_{j=1}^{{K_n}} \|\hat{c}_j\hat{v}_j - v_j\|_{H_1} > \frac{\varepsilon}{4m_n\|\Delta\|_{\cL_{12}}} \Bigg) &\leq \Prob \Bigg( m_n\max_{1\leq j\leq E_n} \|\hat{c}_j\hat{v}_j - v_j\|_{H_1} > \frac{\varepsilon}{4m_n\|\Delta\|_{\cL_{12}}} \Bigg)\\
&\leq \Prob \Bigg( \max_{1\leq j\leq E_n} \frac{2\sqrt{2}}{\hat{\alpha}_j} \|\hat{C} - C\|_{\cL_{12}} > \frac{\varepsilon}{4m^2_n\|\Delta\|_{\cL_{12}}} \Bigg)\\
&\leq \Prob \Bigg( \|\hat{C} - C\|_{\cL_{12}} > \frac{\varepsilon}{8\sqrt{2} m_n^3\|\Delta\|_{\cL_{12}}} \Bigg)\\
&\leq 128 \|\Delta\|_{\cL_{12}}^2  m_n^6 \frac{\E \|\hat{C} - C\|_{\cL_{12}}^2}{\varepsilon^2}\\
&\leq 128U\|\Delta\|_{\cL_{12}}^2 \frac{m_n^6}{\varepsilon^2 n}.
\end{align*}
Since $\hat{\lambda}_{K_n} \geq \frac{1}{m_n}$, the second term in the inequality (\ref{eq:l6}) is bounded by
\begin{align*}
\Prob \Bigg(\lambda_{K_n} < \frac{1}{2m_n}\Bigg) &\leq \Prob \Bigg(\lambda_{K_n} < \frac{1}{2m_n}\ \cap\ |\hat{\lambda}_{K_n} - \lambda_{K_n}| \leq \frac{1}{2m_n} \Bigg) + \Prob \Bigg(|\hat{\lambda}_{K_n} - \lambda_{K_n}| > \frac{1}{2m_n} \Bigg)\\
&\leq \Prob \Bigg(\|\hat{C} - C\|_{\cL_{12}} > \frac{1}{2m_n} \Bigg)\\
&\leq 4m_n^2\E\|\hat{C} - C\|_{\cL_{12}}^2 \leq 4U \frac{m_n^2}{n}.
\end{align*}
Thus we derive
\begin{align*}
\Prob ( \|S_3({K_n})\|_{\cL_{12}} > \varepsilon ) &\leq 128U\|\Delta\|_{\cL_{12}}^2 \frac{m_n^6}{\varepsilon^2 n} + 4U \frac{m_n^2}{n}\leq
U(128\|\Delta\|_{\cL_{12}}^2 + 4\varepsilon^2) \frac{m_n^6}{\varepsilon^2 n}.
\end{align*}
\end{proof}
Finally we need a lemma which assures that $K_n$ tends to infinity.
\begin{lem}\label{lem:kinf}
Let ${K_n}$ be defined as in {\bf (K)}, then $K_n\convP\infty$.
\end{lem}
\begin{proof}
We have to show that $\Prob(\min\{B_n,E_n\} < p) \rightarrow 0$ for any $p\in\mathbb{N}$. Since $\frac{1}{m_n}\searrow 0$, for $n$ large enough we have, by combining Lemma~\ref{Lemma1} and \ref{Lemma3}, that
\begin{align*} 
\Prob(B_n < p) &= \Prob\bigg(\hat{\lambda}_p < \frac{1}{m_n}\bigg) = \Prob\bigg(\lambda_p - \hat{\lambda}_p > \lambda_p - \frac{1}{m_n}\bigg)
\leq \Prob\bigg(|\hat{\lambda}_p - \lambda_p| > \lambda_p - \frac{1}{m_n}\bigg)\to 0.
\end{align*}
\end{proof}
Now we are ready to prove the main result
\begin{proof}[Proof of Theorem \ref{thm:main}]
First, by the triangle inequality we get
\begin{align*}
\|\Psi - \hat{\Psi}_{{K_n}}\|_{\cL_{12}} &\leq \|\Psi-\hat{\Psi}_{{K_n}}\|_{{\cL_{12}}} + \|\Psi - \Psi_{{K_n}}\|_{\cL_{12}}\\
&\leq \|S_1({K_n})\|_{\cL_{12}} + \|S_2({K_n})\|_{\cL_{12}} + \|S_3({K_n})\|_{\cL_{12}} + \|\Psi - \Psi_{{K_n}}\|_{\cL_{12}}.
\end{align*}
By Lemmas \ref{lem:r1}, \ref{lem:d1}, \ref{lem:d2}, \ref{lem:d3} and assumption $m_n^6 = o(n)$ we finally obtain for large enough $n$ that
\begin{align*}
&\Prob( \|\Psi - \hat{\Psi}_{{K_n}}\|_{\cL_{12}} > \varepsilon )\\
&\quad\leq U4^4 \frac{m_n^2}{\varepsilon^4 n} + 4^3U\|\Delta\|^2_{\cS_{12}}\frac{m_n^4}{\varepsilon^2 n} + 4^2U(128\|\Delta\|_{\cL_{12}}^2 + \varepsilon^2/4) \frac{m_n^6}{\varepsilon^2n} + \Prob(\|\Psi - \Psi_{K_n}\|_{\cL_{12}} > \varepsilon/4) \xrightarrow{n \rightarrow \infty} 0.
\end{align*}
\end{proof}
\subsection{Proof of Theorem~\ref{thm:convX}}

In order to simplify the notation we will denote $K=K_n$. This time as a starting point we take a representation of $\Psi$ in the basis $\{\hat v_1, \hat v_2, ...\}$. Let $M_m = \spa\{v_1,v_2,...,v_m\}$, $\hat M_m = \spa\{\hat v_1,\hat v_2,...,\hat v_m\}$ where $\spa\{x_i,\, i\in I\}$ denotes the closed span of the elements $\{x_i,\,i\in I\}$. If $\rank(\hat C)=\ell$, then $\{\hat v_i,\, i> \ell\}$ can be any ONB of $\hat M_\ell^\bot$.   We write $P_A$ for the projection operator which maps on a closed linear space $A$. As usual $A^\bot$ denotes the orthogonal complement of $A$. Since for any $m\geq 1$ we can write $x = P_{\hat M_m}(x) + P_{\hat M_m^\bot}(x)$, the linearity of $\Psi$ and the projection operator gives   
\begin{align*}
\Psi(x) &= \Psi(P_{\hat M_m}(x)) + \Psi(P_{\hat M_m^{\bot}}(x))\\
&= \sum_{j=1}^m\ip{\hat v_j}{x}_{H_1}\Psi(\hat v_j) + \Psi(P_{\hat M_m^{\bot}}(x)).
\end{align*}
Now we evaluate $\Psi$ in some $\hat v_j$ which is not in the kernel of $\hat C$. By definitions of $\Psi$, $\hat C$ and again by linearity of the involved operators
\begin{align*}
\Psi(\hat v_j) &= \frac{1}{\hat \lambda_j} \Psi(\hat C (\hat v_j))\\
&= \frac{1}{\hat \lambda_j} \frac{1}{n} \sum_{i=1}^n \ip{X_i}{\hat v_j}_{H_1} \Psi(X_i)\\
&= \frac{1}{\hat \lambda_j} \frac{1}{n} \sum_{i=1}^n \ip{X_i}{\hat v_j}_{H_1} (Y_i-\varepsilon_i)\\
&= \frac{1}{\hat \lambda_j} (\hat \Delta(\hat v_j) + \hat\Lambda(\hat v_j)),
\end{align*}
where $\hat\Lambda = -\frac{1}{n}\sum_{i=1}^n X_i \otimes \varepsilon_i$. Hence if $m$ is such that $\hat\lambda_m>0$ (which will now be implicitely assumed in the sequel), $\Psi$ can be expressed as
\begin{align*}
\Psi(x) = \sum_{j=1}^m\ip{\hat v_j}{x}_{H_1}\frac{1}{\hat \lambda_j} \hat \Delta(\hat v_j) + \sum_{j=1}^m\ip{ \hat v_j}{x}_{H_1}\frac{1}{\hat \lambda_j} \hat\Lambda (\hat v_j) + \Psi(P_{\hat M_m^\bot}(x)).
\end{align*}
Note that the first term on the right-hand side is just $\hat \Psi_m(x)$. Therefore for any $x$, the distance between $\Psi(x)$ and $\hat \Psi_m(x)$ takes the following form
\begin{align}
\|\Psi(x) - \hat \Psi_m(x)\|_{H_2} = \Bigg\|\sum_{j=1}^m\ip{ \hat v_j}{x}_{H_1}\frac{1}{\hat \lambda_j} \hat\Lambda (\hat v_j) + \Psi(P_{\hat M_m^\bot}(x))\Bigg\|_{H_2}.\label{eq:dpX}
\end{align}
To assess \eqref{eq:dpX} we need the following four lemmas.

\begin{lem}\label{lem:vjconv}
Let $(\lambda_i, v_i)_{i \geq 1}$ and $(\hat \lambda_i, \hat v_i)_{i \geq 1}$ be eigenvalues and eigenfunctions of $C$ and $\hat C$ respectively. Set $j,m \in \mathbb{N}$ such that $j \leq m \leq n$, then
\begin{align*}
\|v_j - P_{\hat  M_m}(v_j)\|_{H_1}^2 \leq 4\frac{\|C-\hat C\|_{\cL_{11}}^2}{(\hat \lambda_{m+1} - \hat \lambda_j)^2}.
\end{align*}
\end{lem}
\begin{proof}
Note that by using Parseval's identity we get
\begin{align*}
\|v_j - P_{\hat M_m}(v_j)\|^2_{H_1} = \sum_{k=1}^\infty \ip{v_j - P_{\hat M_m}(v_j)}{\hat v_k}_{H_1}^2 = \sum_{k>m} \ip{v_j}{\hat v_k}_{H_1}^2.
\end{align*}
Now
\begin{align*}
(\hat \lambda_{m+1} - \hat \lambda_j)^2 \sum_{k>m} \ip{v_j}{\hat v_k}_{H_1}^2
&\leq \sum_{k>m} (\hat \lambda_k \ip{v_j}{\hat v_k}_{H_1} - \hat \lambda_j \ip{v_j}{\hat v_k}_{H_1})^2\\
&= \sum_{k>m} (\ip{v_j}{\hat C(\hat v_k)}_{H_1} - \hat \lambda_j \ip{v_j}{\hat v_k}_{H_1})^2.
\end{align*}
Since $\hat C$ is a self-adjoint operator, simple algebraic transformations yield 
\begin{align*}
(\hat \lambda_{m+1} - \hat \lambda_j)^2 \sum_{k>m} \ip{v_j}{\hat v_k}_{H_1}^2
&\leq \sum_{k>m} (\ip{\hat C(v_j)}{\hat v_k}_{H_1} - \hat \lambda_j \ip{v_j}{\hat v_k}_{H_1})^2\\
&= \sum_{k>m} (\ip{(\hat C-C)(v_j)}{\hat v_k}_{H_1} - (\hat \lambda_j - \lambda_j) \ip{v_j}{\hat v_k}_{H_1})^2\\
&\leq 2 \sum_{k>m} |\ip{(\hat C- C)(v_j)}{\hat v_k}_{H_1}|^2 + 2 \sum_{k>m} ((\hat \lambda_j - \lambda_j) \ip{v_j}{\hat v_k}_{H_1})^2.
\end{align*}
By Parseval's inequality and Lemma~\ref{Lemma3}
\begin{align*}
(\hat \lambda_{m+1} - \hat \lambda_j)^2 \sum_{k>m} \ip{v_j}{\hat v_k}_{H_1}^2
\leq 2 \|(\hat C-C)(v_j)\|_{H_1}^2 + 2 |\hat \lambda_j - \lambda_j|^2
\leq 4 \|\hat C-C\|_{\cL_{11}}^2.
\end{align*}
\end{proof}

\begin{lem}\label{lem:Mnort:conv0}
Let $\Psi$ be defined as in Lemma \ref{thm:convX} and $K=K_n\convP\infty$. Then
$
\|P_{M_{K}^\bot}(X_n)\|_{H_2}\convP 0.
$
\end{lem}

\begin{proof}
We write here and in the sequel $X=X_n$.
We first remark that for any $\varepsilon>0$
\begin{align*}
\Prob(\|P_{M_K^\bot}(X)\|^2_{H_2} > \varepsilon) =\Prob\Bigg(\sum_{i=K+1}^\infty |\ip{v_i}{X}_{H_1}|^2 > \varepsilon\Bigg).
\end{align*}
Since $\sum_{i=1}^\infty |\ip{v_i}{X}_{H_1}|^2 = \|X\|^2_{H_1}$, there exists a random variable $J_\varepsilon \in \mathbb{R}$ such that $\sum_{i=J_\varepsilon}^\infty |\ip{v_i}{X}_{H_1}|^2 < \varepsilon$. Since by assumption $E\|X\|^2_{H_1}<\infty$, we conclude that $J_\varepsilon$ is bounded in probability. Hence we obtain
\begin{align*}
\Prob(\|P_{M_K^\bot}(X)\|^2_{H_2} > \varepsilon) &\leq  \Prob\Bigg(\sum_{i=K+1}^\infty |\ip{v_i}{X}_{H_1}|^2 > \varepsilon\ \cap\ K > J_\varepsilon\Bigg) + \Prob(K \leq J_\varepsilon)\\
&=  \Prob(K \leq J_\varepsilon),
\end{align*}
where the last term converges to zero as $n\to\infty$. 
\end{proof}

\begin{lem}\label{lem:Linf}
Let $L_n = \arg\max\{ r \leq K: \sum_{i=1}^{r} (\hat \lambda_{K+1} - \hat \lambda_i)^{-2} \leq \xi_n\}$, where $K=K_n$ is given as in Theorem~\ref{thm:convX} and $\xi_n\to\infty$. Then
$
L_n\convP\infty.
$
\end{lem}
\begin{proof}
Let $r \in \mathbb{N}$ such that for all $1\leq i\leq r$ we have $\lambda_{r+1} \neq \lambda_i$. Note that $\E\|X\|^2_{H_1} < \infty$ implies $\lambda_i \rightarrow 0$ and since $\lambda_i > 0$ we can find infinitely many $r$ satisfying this condition. We choose such $r$ and obtain
\begin{align*}
\Prob(L_n < r) &\leq \Prob\Bigg(\sum_{i=1}^r \frac{1}{(\hat \lambda_{K+1} - \hat \lambda_i)^2} > \xi_n\ \cap\ {K} \geq r \Bigg) + \Prob({K} < r).
\end{align*}
Lemma \ref{lem:kinf} implies that $\Prob({K} < r) \rightarrow 0$. The first term is bounded by
$\Prob\left(\sum_{i=1}^r \frac{1}{(\hat \lambda_{r+1} - \hat \lambda_i)^2} > \xi_n\ \right)$.
Since $\hat\lambda_i\convP\lambda_i$ and $r$ is fixed while $\xi_n\to\infty$, it follows that $P(L_n<r)\to 0$ if $n\to \infty$. Since $r$ can be chosen arbitrarily large, the proof is finished.
\end{proof}

\begin{lem} \label{lem:Proj}
Let $\Psi$ be defined as in Lemma \ref{thm:convX}, then
$\|P_{M_K}(X) - P_{\hat M_K}(X)\|_{H_1} \convP 0$.
\end{lem}
\begin{proof}
Let us define two variables $X^{(1)} = \sum_{i=1}^L \ip{X}{v_i}_{H_1}v_i$, $X^{(2)} = \sum_{i=L+1}^\infty \ip{X}{v_i}_{H_1}v_i$ and $L$ as in Lemma \ref{lem:Linf}. Again for simplifying the notation we will write $L$ instead of $L_n$. Since $X=X^{(1)}+X^{(2)}$ we derive
\begin{align}
\|P_{M_K}(X) - P_{\hat M_K}(X)\|_{H_1} 
&\leq \|P_{M_K}(X^{(1)}) - P_{\hat M_K}(X^{(1)})\|_{H_1} + \|P_{\hat M_K}(X^{(2)})\|_{H_1} + \|P_{M_K}(X^{(2)})\|_{H_1}. \label{eq:4:3:1}
\end{align}
The last two terms are bounded by $2\|X^{(2)}\|_{H_1}$. For the first summand in \eqref{eq:4:3:1} we get
\begin{align*}
\|P_{M_K}(X^{(1)}) - P_{\hat M_K}(X^{(1)})\|_{H_1} 
&= \Bigg\|\sum_{i=1}^{L}\ip{X}{v_i}_{H_1}(v_i - P_{\hat M_K} (v_i))\Bigg\|_{H_1}.
\end{align*}
Let us choose $\xi_n=o(n)$ in Lemma~\ref{lem:Linf}. 
The triangle inequality, the Cauchy-Schwarz inequality, Lemma~\ref{lem:vjconv} and the definition of $L$ entail
\begin{align*}
&\|P_{M_K}(X^{(1)}) - P_{\hat M_K}(X^{(1)})\|_{H_1}
\leq \sum_{i=1}^{L}|\ip{X}{v_i}_{H_1}|\|v_i - P_{\hat M_K} (v_i)\|_{H_1}\\
&\quad\leq \left(\sum_{i=1}^{L}|\ip{X}{v_i}_{H_1}|^2\right)^{1/2} \left(\sum_{i=1}^{L}\|v_i - P_{\hat M_K} (v_i)\|_{H_1}^2\right)^{1/2}\\
&\quad\leq \|X\|_{H_1} \left(\sum_{i=1}^{L}\|v_i - P_{\hat M_K} (v_i)\|_{H_1}^2\right)^{1/2}\\
&\quad\leq 2\|X\|_{H_1} \|C - \hat C\|_{\cL_{11}} \left(\sum_{i=1}^{L}\frac{1}{(\hat \lambda_{K+1} - \hat \lambda_i)^2}\right)^{1/2}\\
&\quad\leq 2 \|X\|_{H_1} \|C - \hat C\|_{\cL_{11}} \sqrt{\xi_n}.
\end{align*}
This implies the inequality
\begin{align}
\|P_{M_K}(X) - P_{\hat M_K}(X)\|_{H_1} \leq 2 \|X\|_{H_1} \|C - \hat C\|_{\cL_{11}} \sqrt{\xi_n} + 2\|X^{(2)}\|_{H_1}.\label{eq:4:3:2}
\end{align}
Hence by Lemma~\ref{Lemma1} we have $2 \|X\|_{H_1}\|C-\hat C\|_{\cL_{11}}\sqrt{\xi_n}=o_P(1)$. Furthermore we have that $\|X^{(2)}\|=\left(\sum_{j>L}|\langle X,v_j\rangle|^2\right)^{1/2}\convP 0$. This follows from the proof of Lemma~\ref{lem:Mnort:conv0}.
\end{proof}

\begin{lem} \label{lem:Mhat}
Let $\Psi$ be defined as in Lemma \ref{thm:convX}, then $
\|\Psi(P_{\hat M_K^\bot}(X))\|_{H_2}\convP 0.$
\end{lem}
\begin{proof}
Some simple manipulations show
\begin{align*}
\|\Psi(P_{\hat M_K^\bot}(X))\|_{H_2} &= \|\Psi(X - P_{\hat M_K}(X))\|_{H_2}\\
&= \|\Psi(P_{M_K}(X) + P_{M^\bot_K}(X) - P_{\hat M_K}(X))\|_{H_2}\\
&\leq \|\Psi(P_{M_K}(X)) - \Psi(P_{\hat M_K}(X))\|_{H_2}  + \|\Psi(P_{M_K^\bot}(X))\|_{H_2}\\
&\leq\|\Psi\|_{\cL_{12}}\left(\|P_{M_K}(X)- P_{\hat M_K}(X)\|_{H_1}  + \|P_{M_K^\bot}(X)\|_{H_1}\right).
\end{align*}
Direct applications of Lemma~\ref{lem:Mnort:conv0} and Lemma~\ref{lem:Proj} finish the proof.
\end{proof}
\begin{proof}[Proof of Theorem \ref{thm:convX}] Set 
$$
\Theta_n(x)=\sum_{j=1}^{K_n} \frac{\hat\Lambda (\hat v_j)}{\hat \lambda_j}\ip{\hat v_j}{x}_{H_1}.
$$
By the representation \eqref{eq:dpX} and the triangle inequality
\begin{align*}
\|\Psi(X) - \hat\Psi(X)\|_{H_2} \leq \|\Theta_n(X)\|_{H_2} + \|\Psi(P_{\hat M_{K_n}^\bot}(X))\|_{H_2}.
\end{align*}
Lemma~\ref{lem:Mhat} shows that the second term tends to zero in probability. 

If in Lemma \ref{Lemma1} we define $\Psi\equiv 0$, then $\hat\Lambda =\hat \Delta$ and by independence of $\varepsilon_k$ and $X_k$ we get $\Lambda=0$. By the arguments of Lemma \ref{lem:d1} we infer
$
\Prob (\|\Theta_n\|_{\cL_{12}} > \varepsilon) \leq U m_n^2 / \varepsilon^2 n,
$
which implies that $\|\Theta_n(X)\|_{H_2}\convP 0$.
\end{proof}

\section{Acknowledgement}

This research was supported by the Communaut\'e fran\c{c}aise de Belgique---Actions de Recherche Concert\'ees (2010--2015) and  Interuniversity Attraction Poles Programme (IAP-network P7/06), Belgian Science Policy Office.

\providecommand{\bysame}{\leavevmode\hbox to3em{\hrulefill}\thinspace}
\providecommand{\MR}{\relax\ifhmode\unskip\space\fi MR }
\providecommand{\MRhref}[2]{%
  \href{http://www.ams.org/mathscinet-getitem?mr=#1}{#2}
}
\providecommand{\href}[2]{#2}

\begin{addmargin}[1em]{2em}
Corresponding author:
\begin{addmargin}[1em]{1em}
Siegfried H\"ormann\\
D\'epartment de Math\'ematique,\\
Universit\'e libre de Bruxelles (ULB),\\
CP 215, Boulevard du Triomphe\\
B-1050 Bruxelles, Belgium,\\
siegfried.hormann@ulb.ac.be
\end{addmargin}
\end{addmargin}

\end{document}